\documentclass[12pt,a4paper]{article}

\usepackage{amsfonts}
\usepackage{color}
\usepackage{amsmath}
\usepackage{amsbsy}
\usepackage{theorem}
\usepackage{graphicx}
\newtheorem{definition}{Definition}
\newtheorem{theorem}{Theorem}
\newtheorem{proposition}{Proposition}

\newtheorem{remark}{Remark}

\begin{document}
\sf
\title{Sojourn functionals for spatiotemporal Gaussian random fields with long--memory}

\date{}
\author{\normalsize N. Leonenko and M. D. Ruiz--Medina}
% \maketitle
\maketitle

\begin{abstract}This paper addresses the asymptotic analysis of sojourn functionals of spatiotemporal Gaussian random fields with long-range dependence
(LRD) in time also known as long memory. Specifically, reduction theorems are derived for local functionals of nonlinear transformation of such  fields, with Hermite rank $m\geq 1,$ under general covariance structures.   These results are proven to hold, in particular,  for a family of non--separable covariance structures belonging to  Gneiting class.
For $m=2,$ under separability of the spatiotemporal covariance function in space and time,   the properly normalized Minkowski functional, involving the modulus of a Gaussian random field,  converges in distribution   to the Rosenblatt type limiting distribution for a suitable range of the long memory parameter.\end{abstract}

\medskip



\section{Introduction} % Initial capital letter, then lower case. No full stop.

\label{intro}
Geometric characteristics of random surfaces play a critical role in areas such as geostatistics, environmetrics, astrophysics, and medical imaging.
There exists an extensive literature on data analysis based on Gaussian random field modeling. Minkowski functionals have played an important role  in the geometrical analysis of their sample paths.   In Novikov, Schmalzing  and   Mukhanov
 \cite{Novikov00}, Minkowski  functionals are  applied to
the characterization of hot regions (i.e., the excursion sets),
where the normalized temperature fluctuation field exceeds a given
threshold. The normalized temperature fluctuation field, associated
with CMB temperature on the sky, is represented in terms of  a
spherical  random field (see also Linde and  Mukhanov \cite{Linde}; Novikov, Schmalzing  and   Mukhanov \cite{Novikov00}).
 Furthermore, Minkowski functionals are
attractive  due to  their  geometrical interpretation in two
dimensions, in relation to the \rm{total area of all hot regions},
\rm{the total length of the boundary between hot and cold regions},
and the \rm{Euler characteristic}, which counts the number of
isolated hot regions minus the number of isolated cold regions.
Minkowski functionals have also been applied to brain mapping analysis, and, in general, to the description of texture models in medical  imaging analysis, in relation to anatomy segmentation, and pathology detection and diagnosis (see, e.g., Steele \cite{Steele}). Truncated Gaussian processes or sequential indicator simulation play a crucial role in geosciences  to model the spatial distribution of the materials. Here, Minkowski functionals are used as morphological measures (see, e.g., Mosser, Dubrule and Blunt  \cite{MosserDubruleBlunt2017}; Pyrcz and  Deutsch \cite{Pyrz}). In that sense, a wide research area has been developed in the multiscale analysis  of media with complex internal structures (see Armstrong \emph{et al.} \cite{ArmstrongRT}), including soils, sedimentary rocks, foams, ceramics and composite materials
(see, e.g., Gregorov\'a \emph{et al.} \cite{Gregorova}; Ivonin \emph{et al.} \cite{Ivonin}; \cite{Pabst},	and Tsukanov \emph{et al.} \cite{Tsukanov}).  Also, a  good overview and introduction to some of these  applications can be found in Adler and  Taylor \cite{AdlerTaylor07}  and Marinucci and  Peccati \cite{Marinucci}.

   Since the nighties sojour functionals were  extensively analyzed in the context of weak--dependent random fields (see, e.g., Bulinski \emph{et al.} \cite{Bulinski};  Ivanov and Leonenko \cite{Ivanov89}, among others). A parallel literature has also been developed in the long--range dependence random field  context (see  Leonenko \cite{Leonenko99}; Leonenko and Olenko \cite{LeonenkoOlenko14}; Makogin and Spodarev \cite{Makogin20};
Marinucci, Rossi and Vidotto \cite{MarinucciRV},   just to mention a few).  Particularly, limit theorems for level functionals of
stationary Gaussian processes and fields constitute a major topic in this literature (see, e.g., Az\"ais and Wschebor \cite{Wschebor09};  Estrade  and  Le\'on \cite{Estrade16};  Kratz and  Le\'on \cite{KratzLeon}; \cite{KratzLeon2}; Marinucci and Vadlamani \cite{Marinucciv2}). The approach adopted in this paper  continues this research line.

There  has been a growing interest on covariance function modeling for spatiotemporal random fields.  Marinucci, Rossi and Vidotto \cite{MarinucciRV}  consider isotropic in space and stationary in time  Gaussian random fields on
the two--dimensional unit sphere, and investigate the asymptotic behaviour of the empirical measure or excursion area, as time goes to infinity, covering both cases when the underlying field exhibits short and long memory in time. It turns out that the limiting distribution is not universal, depending both of the memory parameter and the threshold or level of sojourn functional. Marinucci, Rossi and Vidotto \cite{MarinucciRV} adopt an  intrinsic spherical isotropic random field   methodology based on  Karhunen-Lo\'eve expansion in terms of  spherical harmonics. As given in their Condition 2, a semiparametric model characterizes the resulting stationary time-varying  angular spectrum involving a memory parameter  depending on the spatial resolution level.  As reflected in  their Condition 3, the smallest exponent corresponding to the largest memory range, and the exponent at the coarsest spatial scale $l=0$ are involved in the scaling to determine   the asymptotic  variance in time of the sojourn functional. Different scenarios are considered,  distinguishing between null and non--null threshold parameter.  Under these scenarios, one can find   the first, second or third chaos domination, respectively leading to Gaussian, and non--Gaussian (so--called composite
Rosenblatt 2 and 3) asymptotic probability distributions.

This paper analyzes the asymptotic behavior in time  of local nonlinear functionals of LRD Gaussian random fields restricted to a spatial  convex compact set. Specifically,  the spectral diagonalization of isotropic continuous covariance kernels on sphere, in terms of spherical harmonics,  applied in  Marinucci, Rossi and Vidotto \cite{MarinucciRV},  is replaced here by the  isonormal representation of a homogeneous and isotropic spatiotemporal Gaussian random field. Our main result, reduction principle given in Theorem \ref{th3}, holds beyond the first Minkowski  functional.   The particular cases of this general reduction principle analyzed in Theorems \ref{th1} and \ref{th2} could be extended to the more general framework of spatial frequency varying long--memory parameters in time, in the spirit of Marinucci, Rossi and Vidotto \cite{MarinucciRV} results. The  same assertions hold regarding   Proposition \ref{p41}  below, derived in a separable covariance  framework in space and time, that will be extended to the non--separable case in a subsequent paper. Note  that the nonseparable covariance modeling assumed in   Marinucci, Rossi and Vidotto \cite{MarinucciRV} is given in terms  of the tensorial product of a spatial basis (spherical harmonics), and a  temporal basis (complex exponentials) that do not provide a diagonal representation. While Proposition \ref{p41} works under the diagonal representation in terms of complex exponentials of the spatiotemporal covariance function of the underlying Gaussian random field.

To focus the topic and better describe the contributions of this work, we have to noting that our starting model is  a spatially homogeneous   and isotropic Gaussian random field, displaying stationarity and LRD  in time, defined on $\mathbb{R}^{d}\times \mathbb{R}.$ Its restriction to  a convex compact set in space is then considered. An increasing sequence of temporal intervals is involved in the  increasing domain asymptotic approach  adopted. Note that our methodology is applicable, in  particular, to considering the  restriction to a  compact two--points homogeneous space,  like the sphere, of our original family of spatiotemporal Gaussian random fields on $\mathbb{R}^{d}\times \mathbb{R}$  (see, e.g., Leonenko and Ruiz--Medina  \cite{Leonenkoetal17a}).

We present a general reduction principle (Theorem \ref{th3}), discovered first by Taqqu \cite{Taqqu} (see also Dobrushin and Major \cite{Dobrushin}; Leonenko, Ruiz--Medina and Taqqu \cite{Leonenkoetal17}; \cite{Leonenkoetal17b}; Taqqu \cite{Taqqu2}), obtaining the limiting distributions of properly normalised integrals of non-linear transformations of spatiotemporal Gaussian random fields, from the asymptotic distribution of Hermite polynomial type functionals of such Gaussian random fields. The method of the proof is standard. Indeed, we use the expansion of the local functional of a Gaussian field into series of Hermite polynomials of such a field. But the novelty of the paper is that we consider spatiotemporal random fields beyond the regularly varying condition on
the spatiotemporal  covariance function.  Hence, we can analyze  a larger class of spatiotemporal covariance functions, including    Gneiting  class (see Gneiting \cite{Gneiting02}). This class of covariance functions is popular in many applications, including Meteorology or Earth sciences, among  others.
 Theorems \ref{th1} shows that, under very general conditions on the decaying of covariance function to zero in time,
the limiting distribution of normalized first Minkowski functional is asymptotically normal for large classes of covariances, including  Gneiting class.
   For the modulus of a Gaussian random field, the limiting distribution is given in the form of a multiple Wiener--It$\widehat{\mbox{o}}$ stochastic integral, assuming separability in  space and time  of the  covariance function. We also assume that the covariance function  is a regularly varying function in time. The derived limiting distribution is of Rosenblatt type.

The outline of the paper is as follows. We first review some results on geometric probabilities in Section \ref{GP}.
 The general reduction theorem, Theorem \ref{th3}, for
subordinated Gaussian spatiotemporal random fields with LRD in time is presented in Section \ref{secrth}. These results are applied to sojourn functionals introduced in Section \ref{sfunct},  providing the asymptotic normality of the first Minkowski functional of a Gaussian random field, and limiting distribution of Rosenblatt type, for the sojourn functional, given by the modulus of a Gaussian random field.
 In Section \ref{scfb} we provide  examples in terms of separable covariance structures. While in  Section \ref{nscf} we present examples of covariance structures for which main results hold for non-separable covariance structures. We restrict out exposition by the covariances known as Gneting class of covariance functions.

\section{Geometric probability}
\label{GP}
Some fundamental elements and basic results on geometric probability are now introduced (see Ahronyan
and Khlatayan   \cite{Aharonyan20};  Ivanov and Leonenko \cite{Ivanov89}; Lellouche and Souries \cite{Souris2020};  Lord \cite{Lord54}, and the references therein).

Let $\nu_{d}$ be the Lebesgue measure on $\mathbb{R}^{d},$  $d\geq 1,$ and $\mathcal{K}$ be a convex body in $\mathbb{R}^{d},$ i.e., a compact convex set with non empty interior. We will denote by $\mathcal{D}(\mathcal{K})= \left\{\max\|\mathbf{x}-\mathbf{y}\|,\ \mathbf{x},\mathbf{y}\in \mathcal{K}\right\}$ the diameter of $\mathcal{K}.$ Let $\nu_{d}(\mathcal{K})=|\mathcal{K}|$ be the volume of $\mathcal{K},$ and  for $d\geq 2,$ $\nu_{d-1}(\delta \mathcal{K})=\mathcal{U}_{d-1}(\mathcal{K})$ be the surface area of $\mathcal{K},$ where $\delta \mathcal{K}$ denotes the boundary of $\mathcal{K}.$  For $d=1,$ we put $\mathcal{U}_{0}(\mathcal{K})=0.$  For example, let $\mathcal{K}=\mathcal{B}(1)=\left\{\mathbf{x}\in \mathbb{R}^{d};\ \|\mathbf{x}\|\leq 1\right\}$ be  the unit ball. Hence,  $\delta \mathcal{K}= \delta \mathcal{B}(1)=\mathcal{S}_{d-1}=\left\{ \mathbf{x}\in \mathbb{R}^{d};\ \|\mathbf{x}\|= 1\right\}$ is the unit sphere. Thus,  \begin{equation}\mathcal{D}(\mathcal{B}(1))=2,\quad \left|\mathcal{B}(1)\right|=\frac{\pi^{d/2}}{\Gamma \left(\frac{d}{2}+1\right)},\ \mathcal{U}_{d-1}\left(\mathcal{B}(1)\right)=\left|\mathcal{S}_{d-1}\right|=\frac{2\pi^{\frac{d}{2}}}{\Gamma \left(\frac{d}{2}\right)}.\label{b1}\end{equation}

Let $\mathcal{Q}$ be the \emph{stellate} space  in $\mathbb{R}^{d},$ and $d\Gamma $ is an element of a locally finite measure in the space  $\mathcal{Q},$
which is invariant with respect to the group $\mathcal{M}$ of all Euclidean motions in the space $\mathbb{R}^{d}.$ Let now consider a chord length distribution function of body $\mathcal{K},$  given by
\begin{equation}
F_{\mathcal{K}}(v)=\frac{2(d-1)}{\left|\mathcal{S}_{d-2}\right|}\int_{\chi(\Gamma )\leq v}d\Gamma,
\label{cld}\end{equation}
\noindent where $\chi(\Gamma )=\Gamma \cap\mathcal{K}$ is  a chord in $\mathcal{K}.$  For example,  if $\mathcal{K}=\mathcal{B}(1),$ then
\begin{eqnarray}
F_{\mathcal{B}(1)}(v)=\left\{\begin{array}{l}0,\quad v\leq 0\\
1-\left(1-\left(\frac{v}{2}\right)^{2}\right)^{\frac{d-1}{2}},\quad 0\leq v\leq 2\\
1,\quad v\geq 2\\
\end{array}\right.
\label{fdb1}
\end{eqnarray}

\noindent (see Ahoronyan
and Khalatyan \cite{Aharonyan20} for details).

  Let now  consider  two  points  $P_{1},P_{2}\in \mathcal{K}$ randomly and independently selected, with uniform distribution in $\mathcal{K}.$
We consider  the probability density $\psi_{\rho_{\mathcal{K}}}$ of the random variable $\rho_{\mathcal{K}}=\left\|P_{1}-P_{2}\right\|,$  given by
 $$\psi_{\rho_{\mathcal{K}}}(z)=\frac{d}{dz}\mathbb{P}\left(\rho_{\mathcal{K}}\leq z\right). $$
 In the particular case $\mathcal{K}=[-1,1],$ $d=1,$ we have  $$\psi_{\rho_{\mathcal{K}}}(u)=1-\frac{u}{2},\quad 0\leq u\leq 2,$$
 \noindent while for $d\geq 2$  (see  Ivanov and Leonenko, 1989; Lord, 1954)
 \begin{equation}
 \psi_{\rho_{\mathcal{B}(1)}}(z)=\mathcal{I}_{1-\left(\frac{z}{2}\right)^{2}}\left(\frac{d+1}{2},\frac{1}{2}\right),\quad 0\leq z\leq 2,
 \label{pdb1}
 \end{equation}

 \noindent where  $\mathcal{I}_{\mu}(p,q)$ denotes the incomplete Beta function, given by
 \begin{equation}\mathcal{I}_{\mu}(p,q)=\frac{\Gamma (p+q)}{\Gamma (p)\Gamma (q)}\int_{0}^{\mu}t^{p-1}(1-t)^{q-1}dt,\quad \mu \in [0,1].
 \end{equation}

 It is also known (see, e.g., equation  (2.6) in Ahoronyan
and Khalatyan \cite{Aharonyan20}) that
\begin{eqnarray}
\psi_{\rho_{\mathcal{K}}}(z)&=&\frac{1}{|\mathcal{K}|^{2}}\left[z^{d-1}\left| \mathcal{S}_{d-1}\right|\left| \mathcal{K}\right|\right.\nonumber\\
&-& \left. z^{d-1}|\mathcal{S}_{d-2}|\mathcal{U}_{d-1}(\mathcal{K})\frac{1}{d-1}\int_{0}^{z}\left(1-F_{\mathcal{K}}(v)\right)
dv\right],\quad 0\leq z\leq \mathcal{D}(\mathcal{K}).\nonumber\\\label{fak20}
\end{eqnarray}

In particular, for the ball  $\mathcal{K}=\mathcal{B}(1),$  we obtain an alternative to equation (\ref{pdb1}), given by,  for $0\leq z\leq 2,$
\begin{equation}
\psi_{\rho_{\mathcal{B}(1)}}(z)=z^{d-1}\left[\frac{2\Gamma \left(\frac{d}{2}+1\right)}{\pi^{\frac{2d-1}{2}}}-\frac{4\Gamma \left(\frac{d}{2}+1\right)}{\pi^{\frac{d-1}{2}}\Gamma \left(\frac{d+1}{2}\right)(d-1)}\int_{0}^{z}\left(1-\left(\frac{u}{2}\right)^{2}\right)^{\frac{d-1}{2}}du\right].
\label{af4}
\end{equation}

 \section{Reduction theorems for spatiotemporal random fields with LRD in time}
 \label{secrth}

 Let $(\Omega ,\mathcal{A},\mathbb{P})$ be the basic probability space, where the random components of the spatiotemporal
real--valued Gaussian random field $\{Z(\mathbf{x},t),\ \mathbf{x}\in \mathbb{R}^{d},\ t\in \mathbb{R}\}$ are defined. That is,
$Z:(\Omega ,\mathcal{A},\mathbb{P})\times \mathbb{R}^{d}\times \mathbb{R}\to \mathbb{R}.$

\medskip

 \noindent \textbf{Condition 1}. Let $Z$ be a measurable mean--square continuous homogeneous and isotropic in space, and stationary in time  Gaussian random field with $\mathbb{E}[Z(\mathbf{x},t)]=0,$ $\mathbb{E}[Z^{2}(\mathbf{x},t)]=1,$ and covariance function $$\widetilde{C}(\|\mathbf{x}-\mathbf{y}\|,|t-s|)=\mathbb{E}\left[ Z(\mathbf{x},t)Z(\mathbf{y},s)\right]\geq 0,\quad \forall t,s\in \mathbb{R},\ \mathbf{x}, \mathbf{y}\in \mathbb{R}^{d}.$$\noindent  In spherical coordinates we denote
 \begin{equation}
 C(z,\tau)=\widetilde{C}(\|\mathbf{x}-\mathbf{y}\|,|t-s|),\quad  z=\|\mathbf{x}-\mathbf{y}\|\geq 0,\  \tau=|t-s|\geq 0.
 \label{lrdcof}
 \end{equation}

 For simplicity, we will use $d\mathbf{x}$ instead of $\nu_{d}(d\mathbf{x}),$ and $dt$ instead of $\nu(dt).$ In the spatiotemporal isotropic spherical random field case  sojourn functionals  have been analyzed in Marinucci, Rossi and Vidotto \cite{MarinucciRV}.  We now introduce the following sojourn functional motivated by the first Minkowski functional.  For each time $t$ fixed, the random area
 \begin{eqnarray}
 \mathcal{A}_{u}(t)&=&\left|Z^{-1}(\cdot,t)\left( [u,\infty)\right)\right|=\left|\left\{\mathbf{x}\in \mathcal{K};\ Z(\mathbf{x},t)\geq u \right\}\right|
\nonumber\\ & =&\int_{\mathcal{K}}1_{\mathcal{S}_{Z}(u)}(\mathbf{x},t)d\mathbf{x}\nonumber
 \label{ramf}
 \end{eqnarray}
 \noindent provides the empirical measure (i.e., the excursion area) of $Z(\cdot,t)$ corresponding to the level $u,$ $u\in \mathbb{R}.$  The integrated area over the temporal interval $[0,T]$ is then computed as
  \begin{equation}
 M_{T}^{(1)}(u)=\left|\left\{0\leq t\leq T;\  \mathbf{x}\in \mathcal{K},\  Z(\mathbf{x},t)\geq u\right\}\right|=\int_{0}^{T}\int_{\mathcal{K}}1_{\mathcal{S}_{Z}(u)}
 (\mathbf{x},t)d\mathbf{x}dt,
 \label{mff}
 \end{equation}
 \noindent where $1_{\mathcal{S}_{Z}(u)}(\cdot,\cdot)$ denotes the indicator function of the set  \linebreak $\mathcal{S}_{Z}(u)=\{(\mathbf{y},s)\in \mathcal{K}\times [0,T];\ Z(\mathbf{y},s)\geq u\}.$  Similarly, we can define, for $u\geq 0,$ the random area
  \begin{eqnarray}
 \mathcal{\widetilde{A}}_{u}(t)&=&\left|Z^{-1}(\cdot,t)\left[ (-\infty,-u]\cup [u,\infty)\right]\right|=\left|\left\{\mathbf{x}\in \mathcal{K};\ \left| Z(\mathbf{x},t)\right|\geq u \right\}\right|\nonumber\\
 &=&\int_{\mathcal{K}}1_{\mathcal{S}_{|Z|}(u)}(\mathbf{x},t)d\mathbf{x},\nonumber
 \label{ramf2}
 \end{eqnarray}
\noindent temporally integrated over $[0,T],$  defining the functional
  \begin{equation}
 M_{T}^{(2)}(u)=\left|\left\{0\leq t\leq T;\ \mathbf{x}\in \mathcal{K},\ \left|Z(\mathbf{x},t)\right|\geq u \right\}\right|=\int_{0}^{T}\int_{\mathcal{K}}1_{\mathcal{S}_{|Z|}(u)}
 (\mathbf{x},t)d\mathbf{x}dt.
 \label{mff2}
 \end{equation}

  Let $Z\sim \mathcal{N}(0,1)$ be a standard Gaussian random variable with probability density $\phi ,$ and distribution  function  $\Phi $  given by
  $$\phi (z)=\frac{1}{\sqrt{2\pi}}\exp\left(-\frac{z^{2}}{2}\right),\quad  \Phi (u)=\int_{-\infty }^{u}\phi_{Z}(z)dz,\quad z,u\in \mathbb{R}.$$

  Let now $\mathcal{G}$ be a Borel measurable function such that $$\int_{\mathbb{R}}[\mathcal{G}(z)]^{2}\phi (z)dz<\infty.$$
  Then, $\mathcal{G}$ has an expansion with respect to the normalized Hermite polynomials that converges in $L_{2}(\mathbb{R}, \phi (z)dz ):$

  \begin{equation}\mathcal{G}(z)=\sum_{q=0}^{\infty }\frac{\mathcal{G}_{q}}{q!}H_{q}(z),\quad z\in \mathbb{R},\quad \mathcal{G}_{q}=\int_{\mathbb{R}}H_{q}(\xi)\mathcal{G}(\xi)\phi(\xi)d\xi,\quad  q\geq 1,\label{Hefind}\end{equation}

\noindent where the Hermite polynomial of order $q\geq 1,$  denoted as $H_{q}$ satisfies the equation:
\begin{equation}
\frac{d^{n}\phi}{dz^{n}}(z)=(-1)^{n}H_{n}(z)\phi(z).
\label{hq}
\end{equation}

  Note that  \begin{eqnarray}
H_{0}(x)&=& 1,\ H_{1}(x)=x,\ H_{2}(x)=x^{2}-1\nonumber\\
H_{3}(x)&=& x^{3}-3x,\ H_{4}(x)= x^{4}-6x^{2}+3,\dots \boldsymbol{.}\label{exherpol}
\end{eqnarray}

  Particularly, if $\mathcal{G}_{u}(z)=1_{\{z\geq u\}},$  we then obtain

  \begin{eqnarray}\mathcal{G}_{u}(Z(\mathbf{x},t))&=& \mathbb{E}[1_{\mathcal{S}_{Z}(u)}(\mathbf{x},t)]\nonumber\\
  &+&\sum_{q=1}^{\infty }\frac{\mathcal{G}_{q}(u)}{q!}H_{q}(Z(\mathbf{x},t)),\quad \forall (\mathbf{x},t)\in \mathbb{R}^{d}\times \mathbb{R}.\label{Hefind}\end{eqnarray}

Here,  for every $(\mathbf{x},t)\in \mathbb{R}^{d}\times \mathbb{R},$
\begin{eqnarray}\mathcal{G}_{0}(u) &=&
\mathbb{E}\left[1_{ \mathcal{S}_{Z}(u) }(\mathbf{x},t)\right] = [1-\Phi(u)]=\int_{u }^{\infty}\phi(\xi)d\xi\nonumber\\
\mathcal{G}_{q}(u)&=& \phi(u)H_{q-1}(u),\quad q\geq 1.
\label{hpc}
\end{eqnarray}

  For the second functional corresponding to $\widetilde{\mathcal{G}}_{u}(z)=1_{\{\left|z\right|\geq u\}},$
  we have
  \begin{eqnarray}\widetilde{\mathcal{G}}_{0}(u) &=&
\mathbb{E}[1_{\mathcal{S}_{|Z|}(u)}(\mathbf{x},t)] =2 [1-\Phi(u)]=2\int_{u}^{\infty}\phi(\xi)d\xi\nonumber\\
\widetilde{\mathcal{G}}_{q}(u)&=& 2\phi(u)H_{q-1}(u),
\label{hpc2}
\end{eqnarray}\noindent for any even $q\geq 0,$   and $\widetilde{\mathcal{G}}_{q}(u)=0,$ for odd $q\geq 1.$

In what follows, from (\ref{Hefind})--(\ref{hpc2}), we will consider the induced expansions of the functionals $M_{T}^{(i)}(u),$ $i=1,2,$
given by
\begin{eqnarray}
M_{T}^{(1)}(u)&=& (1-\Phi(u))T|\mathcal{K}|+\phi(u)\sum_{n=1}^{\infty}\frac{H_{n-1}(u)}{n!}\eta_{n}\label{hemf1}\\
M_{T}^{(2)}(u) &=& 2(1-\Phi(u))T|\mathcal{K}|+2\phi(u)\sum_{n=1}^{\infty}\frac{H_{2n-1}(u)}{(2n)!}\eta_{n},
\label{hemf}
\end{eqnarray}
\noindent where
%%%%%%%%%%%%%%%%%ACLARAR SI SE ESTÁ SUPONIENDO QUE LA COV ES POSITIVA%%%%%%%%%%%%%%%%%
\begin{eqnarray}
\eta_{n}&=&\int_{0}^{T}\int_{\mathcal{K}}H_{n}(Z(\mathbf{x},t))d\mathbf{x}dt,\nonumber\end{eqnarray}
\noindent and
\begin{eqnarray}
\mathbb{E}[\eta_{n}]&=&0,\quad \mathbb{E}[\eta_{n}\eta_{l}]=0,\quad n\neq l,\nonumber\\
\sigma^{2}_{n,\mathcal{K}}(T)&=&\mathbb{E}[\eta_{n}^{2}]=2n!T\int_{0}^{T}\left(1-\frac{\tau}{T}\right)\int_{\mathcal{K}\times  \mathcal{K}}
\widetilde{C}^{n}(\|\mathbf{x}-\mathbf{y}\|,\tau)d\tau  d\mathbf{x}d\mathbf{y}\nonumber\\
&=&2n!T|\mathcal{K}|^{2}\int_{0}^{T}\left(1-\frac{\tau}{T}\right)\mathbb{E}\left[\widetilde{C}^{n}\left(\|P_{1}-P_{2}\|,\tau\right) \right]d\tau
\nonumber\\
&=&2n!T|\mathcal{K}|^{2}\int_{0}^{T}\left(1-\frac{\tau}{T}\right)\int_{0}^{\mathcal{D}(\mathcal{K})}\psi_{\rho_{\mathcal{K}}}(z)C^{n}(z,\tau )dzd\tau.
\label{rvex}
\end{eqnarray}

\medskip

\noindent \noindent \textbf{Condition 2}.  Assume that
\begin{itemize}
\item[(i)]  $\sup_{z\in [0,\mathcal{D}(\mathcal{K})]}|C(z,\tau)|=\sup_{z\in [0,\mathcal{D}(\mathcal{K})]}C(z,\tau)\to 0,$  $\tau\to \infty$
\item[(ii)]  For certain fixed $m\in \{ 1,2,\dots\},$  there exists $\delta \in (0,1)$ such that
\begin{eqnarray}
\lim_{T\to \infty}\frac{1}{T^{\delta }}\int_{0}^{T}\left(1-\frac{\tau}{T}\right)\int_{0}^{\mathcal{D}(\mathcal{K})}C^{m}(z,\tau)\psi_{\rho_{\mathcal{K}}}(z)dzd\tau=\infty.
\nonumber\\
\label{con}
\end{eqnarray}
\end{itemize}

 \subsection{Reduction theorem}
 In this section, we extend the results by  Taqqu  \cite{Taqqu};\cite{Taqqu2} to the case of spatiotemporal random fields with LRD in time. For a function $\mathcal{G}\in L_{2}(\mathbb{R},\phi(u)du),$  under \textbf{Condition 1}, we consider the following local functional
 \begin{eqnarray}
 A_{T}&=&\int_{0}^{T}\int_{\mathcal{K}}\mathcal{G}(Z(\mathbf{x},t))d\mathbf{x}dt\nonumber\\
 &=& T|\mathcal{K}|\mathcal{G}_{0}+\sum_{n=1}^{\infty}\frac{\mathcal{G}_{n}}{n!}\int_{0}^{T}\int_{\mathcal{K}}H_{n}(Z(\mathbf{x},t))d\mathbf{x}dt,
 \label{HESTF}
 \end{eqnarray}
 \noindent where, for $n\geq 0,$  $\mathcal{G}_{n}$  denotes the Fourier coefficient of function $\mathcal{G}$ with respect to $H_{n},$ and the series (\ref{HESTF}) converges in $\mathcal{L}_{2}(\Omega,\mathcal{A},P).$ Denote as in (\ref{rvex}), $\sigma^{2}_{n,\mathcal{K}}(T)=\mathbb{E}[\eta_{n}^{2}],$ hence, we obtain
 \begin{eqnarray}
 \sigma_{T}^{2}=\mbox{Var}(A_{T})=\mathbb{E}\left[A_{T}-E[A_{T}]\right]^{2}=\sum_{n=0}^{\infty }\sigma^{2}_{n,\mathcal{K}}(T).\nonumber
 % \label{HESTF2}
 \end{eqnarray}

 \medskip

 \noindent \textbf{Definition}.   We say that an integer $m\geq 1$  is the Hermite rank of function $\mathcal{G},$  if for $m=1,$
 $\mathcal{G}_{1}\neq 0,$ or for $m\geq 2,$  $\mathcal{G}_{1}=\dots =\mathcal{G}_{m-1}=0,$ $\mathcal{G}_{m}\neq 0$ (see also Taqqu \cite{Taqqu}).

 \begin{theorem}
  \label{th3}
  Under \textbf{Conditions 1} and \textbf{2}, assume that function $\mathcal{G}$ in  (\ref{HESTF}) has Hermite rank $m,$ the random variables
  \begin{eqnarray}
  \label{rth3}
  Y_{T}=\frac{A_{T}-\mathbb{E}[A_{T}]}{|\mathcal{G}_{m}|\sigma_{m,\mathcal{K}}(T)(1/m!)}
  \end{eqnarray}
  \noindent and
  \begin{eqnarray}
  Y_{m,T}=\frac{\mbox{sgn}\{\mathcal{G}_{m}\}\int_{0}^{T}\int_{\mathcal{K}}H_{m}(Z(\mathbf{x},t))d\mathbf{x}dt}{\sigma_{m,\mathcal{K}}(T)}
  \end{eqnarray}
  \noindent have the same limiting distributions (if one of it exists).
 \end{theorem}
 \begin{proof}
 We split
 $$A_{T}-\mathbb{E}[A_{T}]=S_{1,T}+S_{2,T},$$
 \noindent where using notation (\ref{rvex}), and applying Parseval identity,
 \begin{eqnarray}
 S_{1,T}=\frac{\mathcal{G}_{m}}{m!}\xi_{m},\quad  S_{2,T}=\sum_{n=m+1}^{\infty}\frac{\mathcal{G}_{n}}{n!}\xi_{n},\quad \sum_{n=m}^{\infty }
 \frac{\mathcal{G}^{2}_{n}}{n!}<\infty\quad \mbox{a.s}.
 \label{hefm}
 \end{eqnarray}
 From (\ref{rvex}) and (\ref{hefm}) we get
 \begin{equation}
 \mbox{Var}\left( A_{T}\right)=\mbox{Var}(S_{1,T})+\mbox{Var}(S_{2,T}),
 \label{varfunc}
 \end{equation}
 \noindent and we have to show that $$\frac{\mbox{Var}(S_{2,T})}{\sigma^{2}_{m,\mathcal{K}}(T)}\to  0,\quad T\to \infty.$$

 Under \textbf{Condition 2}(i),
 \begin{equation}
 \sup_{z\in [0,\mathcal{D}(\mathcal{K})], \tau\geq T^{\delta}} C(z,\tau )\to 0,\quad T\to \infty,
 \label{cobzer}
 \end{equation}
 \noindent where $\delta $ satisfies \textbf{Condition} 2(ii). Note that,  for $0\leq \tau\leq T^{\delta},$ the unit variance of $Z$ allows to work with the uniform estimate
 $|C(z,\tau)|^{m+1}\leq 1,$ $z\in \mathbb{R}_{+}.$  From (\ref{hefm}), we then have
 %%%%%%%%%%%ACLARAR SI SE ESTÁN CONSIDERANDO COV POSITIVAS EN (\ref{rvex})%%%%%%%
 \begin{eqnarray}
 \mbox{Var}(S_{2,T})&\leq &\sum_{n=m+1}^{\infty}\frac{\mathcal{G}_{n}^{2}}{(n!)^{2}}\sigma_{n,\mathcal{K}}^{2}(T)\nonumber\\
 &\leq & M_{1}\left\{2T\left[\int_{0}^{T^{\delta }}+\int_{T^{\delta }}^{T}\right]\right\}\left(1-\frac{\tau}{T}\right)\int_{0}^{\mathcal{D}(\mathcal{K})}C^{m+1}(z,\tau)
 \psi_{\rho_{\mathcal{K}}}(z)dzd\tau,\nonumber\\
 \label{bub}
 \end{eqnarray}
 \noindent for $M_{1}>0, $ whose value follows from (\ref{rvex}). In addition, from (\ref{fak20}),
 \begin{equation}
 \psi_{\rho_{\mathcal{K}}}(z)\leq \frac{z^{d-1}}{|\mathcal{K}|}|S_{d-1}|,\quad 0\leq z\leq \mathcal{D}(\mathcal{K}),\label{ine31}
 \end{equation}
 \noindent leading to
 \begin{eqnarray}
 \mbox{Var}(S_{2,T})&\leq & M_{1}\left\{M_{2}T^{\delta +1}+2T\int_{T^{\delta }}^{T}\left(1-\frac{\tau}{T}\right)\int_{\mathcal{K}}C^{m+1}(z,\tau)
 \psi_{\rho_{\mathcal{K}}}(z)dzd\tau\right\}\nonumber\\
 &\leq &M_{3}\left\{T^{\delta +1}+2T\sup_{z\in [0,\mathcal{D}(\mathcal{K})], \tau \geq T^{\delta }}C(z,\tau)\right.\nonumber\\
 &&\hspace*{2cm}\left.\times \int_{T^{\delta }}^{T}\left(1-\frac{\tau}{T}\right)\int_{\mathcal{K}}C^{m}(z,\tau)
 \psi_{\rho_{\mathcal{K}}}(z)dzd\tau\right\}.
 \label{eth3in2}
 \end{eqnarray}
 Hence,
 \begin{eqnarray}
&&\hspace*{-2cm}\frac{ \mbox{Var}(S_{2,T})}{\sigma_{m,\mathcal{K}}^{2}(T)}\leq M_{4}\left\{\frac{1}{T^{-(\delta +1)}\sigma_{m,\mathcal{K}}^{2}(T)}\right.\nonumber\\ &&\hspace*{0.5cm}\left.+M_{5}\sup_{z\in [0,\mathcal{D}(\mathcal{K})], \tau\geq T^{\delta}}C(z,\tau)\frac{\int_{T^{\delta }}^{T}\left(1-\frac{\tau}{T}\right)\int_{\mathcal{K}}C^{m}(z,\tau)
 \psi_{\rho_{\mathcal{K}}}(z)dzd\tau}{\int_{0}^{T}\left(1-\frac{\tau}{T}\right)\int_{\mathcal{K}}C^{m}(z,\tau)
 \psi_{\rho_{\mathcal{K}}}(z)dzd\tau}\right\}.\nonumber\\
  \label{eth3in2b}
 \end{eqnarray}
 From  (\ref{rvex}), under \textbf{Condition 2}(ii),
 \begin{eqnarray}
 \frac{\sigma_{m,\mathcal{K}}^{2}(T)}{T^{\delta +1}}\to \infty,\quad T\to \infty,
  \label{sc2}
 \end{eqnarray}
 \noindent and under \textbf{Condition 2(i)},
 $$\sup_{z\in  [0,\mathcal{D}(\mathcal{K})], \tau\geq T^{\delta}}C(z,\tau)\to 0,\quad T\to \infty.$$
 %%%%%%%%%%%%%%%%%DE NUEVO SE SUPONE QUE LA CORRELACIÓN ES POSITIVA%%%%%%%%%%%%%%%%
 Note that, \begin{equation} \frac{\int_{T^{\delta }}^{T}\left(1-\frac{\tau}{T}\right)\int_{\mathcal{K}}C^{m}(z,\tau)
 \psi_{\rho_{\mathcal{K}}}(z)dzd\tau}{\int_{0}^{T}\left(1-\frac{\tau}{T}\right)\int_{\mathcal{K}}C^{m}(z,\tau)
 \psi_{\rho_{\mathcal{K}}}(z)dzd\tau}\leq 1.  \label{c1cov}
 \end{equation}

 The  convergence to zero of  $\frac{ \mbox{\small Var}(S_{2,T})}{\sigma_{m,\mathcal{K}}^{2}(T)}$  then follows from equation (\ref{eth3in2b}) under \textbf{Condition 2}, leading to  $\mathbb{E}[Y_{T}-Y_{m,T}]^{2}=\frac{ \mbox{\small Var}(S_{2,T})}{\sigma_{m,\mathcal{K}}^{2}(T)}\to 0,$  $T\to \infty,$ as we wanted to prove.
 \end{proof}

 \begin{remark}
 \label{rem1}
 We have applied  in (\ref{c1cov}) that, under \textbf{Condition 1}, the correlation function $C(z,\tau)\geq 0,$ for every $\tau ,z\in \mathbb{R}_{+}.$

 \end{remark}

\begin{remark}

For a short memory case, one can assume that for a fixed $m\geq 1,$
$$\int_{0}^{\infty}\int_{0}^{\mathcal{D}(\mathcal{K})} \psi_{\rho_{\mathcal{K}}}(\mathbf{z})\left|C(\mathbf{z},\tau)\right|^{m}d\mathbf{z}d\tau <\infty.$$
Then, once can show, using standard arguments that,  as $T\to \infty,$ the asymptotic variance satisfies
\begin{equation}
\sigma_{T}^{2}=\mbox{Var}(A_{T})=\sum_{n=m}^{\infty }\frac{\mathcal{G}_{n}^{2}}{(n!)^{2}}\sigma^{2}_{n,\mathcal{K}}(T)=TB(1+o(1)),
\label{er2}
\end{equation}

\noindent where $$B=\sum_{n=m}^{\infty}\frac{\mathcal{G}_{n}^{2}}{(n!)^{2}}\lim_{T\to \infty}\frac{\sigma^{2}_{n,\mathcal{K}}(T)}{T}<\infty.$$

Then, using the method of moments and diagram formulae (see for details Theorem 2.3.1 in Ivanov and Leonenko \cite{Ivanov89}), one can prove under condition
(\ref{er2}) that   $(A_{T}-E[A_{T}])/\sqrt{T}$ converges to a normal distribution with zero mean and variance $B.$
 \end{remark}

\section{Sojourn functionals}
\label{sfunct}

As an application of reduction  Theorem  \ref{th3}, the following result proves  the convergence to a standard normal distribution for the case of Hermite rank $m=1.$

\begin{theorem}
\label{th1}
Under \textbf{Conditions 1} and \textbf{2}, for $m=1,$ the random variables
\begin{eqnarray}
&&X_{1,T}=\frac{M_{T}^{(1)}(u)-T|\mathcal{K}|(1-\Phi(u))}{\phi(u)\left[2T\int_{0}^{T}\left(1-\frac{\tau }{T}\right)\int_{0}^{\mathcal{D}(\mathcal{K})} C(z,\tau)\psi_{\rho_{\mathcal{K}}}(z)dzd\tau\right]^{1/2}},\nonumber\\
\label{rvthr}\end{eqnarray}
\noindent and
\begin{eqnarray}
\frac{\int_{0}^{T}\int_{\mathcal{K}}Z(x,t)dxdt}{\left[2T\int_{0}^{T}\left(1-\frac{\tau }{T}\right)\int_{0}^{\mathcal{D}(\mathcal{K})}C(z,\tau)\psi_{\rho_{\mathcal{K}}}(z)dzd\tau\right]^{1/2}}
\label{rvthr2}
\end{eqnarray}
\noindent have the same limit as $T\to \infty.$  Namely,  the convergence to  a standard normal distribution holds.

\end{theorem}

 The analogous result to Theorem \ref{th1} for functional $M_{T}^{(2)}$ is now formulated.
 \begin{theorem}
 \label{th2}
 Under \textbf{Conditions 1} and \textbf{2}, with $m=2,$  the random variables
  \begin{eqnarray}
 &&X_{2,T}=\frac{M_{T}^{(2)}(u)-2T|\mathcal{K}|(1-\Phi(u))}{[\phi(u)]^{2}\left[T|\mathcal{K}|^{2}\int_{0}^{T}\left(1-\frac{\tau }{T}\right)\int_{0}^{\mathcal{D}(\mathcal{K})} C^{2}(z,\tau)\psi_{\rho_{\mathcal{K}}}(z)dzd\tau\right]^{1/2}},\nonumber\\
 \label{rvth2}
 \end{eqnarray}
 \noindent and
 \begin{eqnarray}
&&Y_{2,T}=\frac{\int_{0}^{T}\int_{\mathcal{K}}(Z^{2}(x,t)-1)dxdt}{2\left[T|\mathcal{K}|^{2}\int_{0}^{T}\left(1-\frac{\tau }{T}\right)\int_{0}^{\mathcal{D}(\mathcal{K})}C^{2}(z,\tau)\psi_{\rho_{\mathcal{K}}}(z)dzd\tau\right]^{1/2}}
\label{rvthr2b}
\end{eqnarray}
 \noindent have the same limit distribution in the sense that if one exists then so does the other and the two are equal.
 \end{theorem}

 The proof of Theorem \ref{th2} is obtained from Theorem  \ref{th3}.

\section{Examples}
\label{ex}
In Sections \ref{scfb} and \ref{nscf}, we  will present  some examples of covariance functions displaying long--range dependence in time for which Conditions 2(i)--(ii) hold true.
 \subsection{Separable covariance structures}
\label{scfb}
 Under \textbf{Condition 1}, the covariance function $\widetilde{C}(\|\mathbf{x}-\mathbf{y}\|,|t-s|)=C(z,\tau)$ is said to be separable if it can
 be factorized as the product of a spatial  $C_{\mathcal{S}}$ and  temporal $C_{\mathcal{T}}$ covariance functions (see Cressie and Huang \cite{Cressie}, and Christakos  \cite{Christakos00}). That is,
 \begin{equation}\widetilde{C}(\|\mathbf{x}-\mathbf{y}\|,|t-s|)=C(z,\tau)= C_{\mathcal{S}}(z)C_{\mathcal{T}}(\tau ),\label{scf}
 \end{equation}
 \noindent where, as before, $z\geq 0,$ and $\tau\geq 0.$

 \medskip

 \noindent \textbf{Condition 4}.  Consider the covariance function
 \begin{equation}
 C_{\mathcal{T}}(\tau )= \frac{\mathcal{L}(\tau )}{\tau^{\alpha }},\quad \tau \geq 0,\quad \alpha \in (0,1),
 \label{eqscm1}
 \end{equation}
 \noindent where $\mathcal{L}$ is a slowly varying function at infinity  locally bounded, i.e., bounded at each bounded interval.

 Under \textbf{Conditions 1} and \textbf{4}, for $\alpha \in \left(0,\frac{1}{n}\right),$ for separable covariance functions as given in (\ref{scf}), we obtain
 \begin{eqnarray}
 \sigma_{n}^{2}(T)&=&2n! T\int_{0}^{T}\left(1-\frac{\tau }{T}\right)C^{n}_{\mathcal{T}}(\tau )d\tau\nonumber\\
&=& T^{2-n\alpha }\mathcal{L}^{n}(T)\left[2n!\int_{0}^{1}(1-\tau)\tau^{-n\alpha }d\tau\right] \left( 1+o(1)\right).
 \label{varscf}
 \end{eqnarray}
 \noindent    From (\ref{rvex}) and (\ref{varscf}), as $T\to \infty,$
 \begin{eqnarray}
 \sigma_{n,\mathcal{K}}^{2}(T)&=&c_{\mathcal{K}}(n,\alpha )T^{2-n\alpha }\mathcal{L}^{n}(T)\left( 1+o(1)\right),\nonumber\end{eqnarray}
 \noindent where $$c_{\mathcal{K}}(n,\alpha )= 2n!\left[\int_{0}^{1}\left(1-\tau \right)\frac{d\tau }{\tau^{\alpha n}}\right]|\mathcal{K}|^{2}\int_{0}^{\mathcal{D}(\mathcal{K})} C_{\mathcal{S}}(z)\psi_{\rho_{\mathcal{K}}}(z)dz.$$

 \begin{proposition}
 \label{p41}
 Under \textbf{Conditions 1} and \textbf{4}, for separable covariance functions  (\ref{scf}), \textbf{Condition 2(ii)} holds for $\alpha \in (0,1),$
 if $m=1,$ and for  $\alpha \in (0,1/2)$ if $m=2.$  Moreover, for $\alpha \in (0,1/2),$ the random variables (\ref {rvth2}) and (\ref{rvthr2b}) have,
  as $T\rightarrow \infty ,$   the limiting  distribution  $\mathcal{R}$ of Rosenblatt type, given
 by the following Wiener-It$\widehat{\mbox{o}}$ integral representation, with respect to spatiotemporal complex Gaussian white noise random measure $W$
on $\mathbb{R}^{2}\times \mathbb{R}^{2d}$ (integration over hyperdiagonals are excluded, see, e.g., Dobrushin and Major \cite{Dobrushin})
\begin{eqnarray}
\mathcal{R}&=&\frac{c_{T}(\alpha )}{\sqrt{c_{\mathcal{K}}(2,\alpha )}}\int_{\mathbb{R}^{2}}^{\prime }\frac{\exp \{i(\mu _{1}+\mu _{2})\}-1}{i(\mu _{1}+\mu _{2})}\frac{1}{\left\vert
\mu _{1}\mu _{2}\right\vert ^{\frac{1-\alpha }{2}}}\nonumber\\
&\times & \int_{\mathbb{R}^{2d}}^{\prime }\left[\int_{\mathcal{K}}\exp \left\{i\left\langle \mathbf{x},\boldsymbol{\omega}_{1}+\boldsymbol{\omega}_{2}\right\rangle\right\}d\mathbf{x}\right]\left[\prod_{j=1}^{2}
f_{S}(\boldsymbol{\omega} _{j})\right]^{1/2}\nonumber\\
&&\hspace*{3cm}\times W(d\mu
_{1},d\boldsymbol{\omega} _{1})W(d\mu _{2},d\boldsymbol{\omega} _{2}),
\label{srncl}
\end{eqnarray}
\noindent where
$$f_{\mathcal{S}}(\boldsymbol{\omega })=\frac{1}{[2\pi]^{d}}\int_{\mathbb{R}^{d}}\exp\left(-i\left\langle \boldsymbol{\omega },\mathbf{x}\right\rangle\right)  \widetilde{C}_{\mathcal{S}}(\|\mathbf{x}\|) d\mathbf{x},$$
 \noindent and
\begin{equation}c_{T}(\alpha )=\frac{\Gamma \left(\frac{1-\alpha }{2}\right)}{2^{\alpha }\Gamma \left(\frac{\alpha }{2}\right)\sqrt{\pi}}\label{TTC}
 \end{equation}
\noindent is the  Tauberian constant.

 \end{proposition}
 \begin{remark}
 Note that
$E\left[\mathcal{R}^{2}\right]<\infty .
$

 \end{remark}
 \begin{proof}
 The proof of Proposition \ref{p41} is standard (see, e.g., Leonenko and Olenko \cite{LeonenkoOlenko14}). A sketch of the proof is now given.
 Note that the spectral density of a spatiotemporal random field with separable covariance function  (\ref{scf}) is also separable, i.e.,
 \begin{eqnarray}
 f(\boldsymbol{\omega },\mu)&=& \frac{1}{(2\pi)^{d+1}}\int_{\mathbb{R}\times \mathbb{R}^{d}}\exp\left(-i\mu\tau\right)
 \exp\left(-i\left\langle \boldsymbol{\omega },\mathbf{x}\right\rangle\right)\widetilde{C}_{\mathcal{S}}(\|\mathbf{x}\|)\widetilde{C}_{\mathcal{T}}(|\tau|)d\mathbf{x}d\tau\nonumber\\
 &=&\left[\frac{1}{[2\pi]^{d}}\int_{\mathbb{R}^{d}}\exp\left(-i\left\langle \boldsymbol{\omega },\mathbf{x}\right\rangle\right)  \widetilde{C}_{\mathcal{S}}(\|\mathbf{x}\|) d\mathbf{x}\right]\left[\frac{1}{2\pi}\int_{\mathbb{R}}\exp\left(-i\mu\tau\right)\widetilde{C}_{\mathcal{T}}(|\tau|)d\tau\right]\nonumber\\
 &=& f_{\mathcal{S}}(\boldsymbol{\omega })f_{\mathcal{T}}(\mu ),\quad \boldsymbol{\omega }\in \mathbb{R}^{d},\quad
 \mu \in \mathbb{R}.\label{spsd}
 \end{eqnarray}

 From Tauberian Theorems (see Leonenko and Olenko \cite{LeonenkoOlenko13}), under \textbf{Condition 4}, we get  convergence
 \begin{equation}
 f_{\mathcal{T}}(\mu )\sim c_{T}(\alpha )\frac{\mathcal{L}\left(\frac{1}{\mu}\right)}{|\mu|^{1-\alpha }},\quad \mu\to 0,\label{TT}
 \end{equation}
 \noindent for $0<\alpha <\frac{1}{2},$  where the Tauberian constant $c_{T}(\alpha )$ has been introduced in (\ref{TTC}). From (\ref{spsd}),  applying the  Wiener--It$\widehat{\mbox{o}}$ stochastic integral representation (see, e.g., Major \cite{Major81}, and Section 4.4.2 in Marinucci and  Peccati \cite{Marinucci}), we obtain isonormal representation:
 \begin{equation}
 Z(\mathbf{x},t)=\int_{\mathbb{R}^{d}}\int_{\mathbb{R}}\exp\left(i\mu t\right)\exp\left(i\left\langle \boldsymbol{\omega},\mathbf{x}\right\rangle\right)\sqrt{f_{\mathcal{T}}(\mu )f_{\mathcal{S}}(\boldsymbol{\omega })}W(d\mu,d\boldsymbol{\omega }),\label{wtr}\end{equation}
 \noindent with $W$ denoting complex--valued white noise measure.

For $\underset{d}{=}$ denoting the identity in probability distribution,     applying now the self--similarity of Gaussian white noise random measure
 \begin{eqnarray}&&
 W(a d\mu,b d\boldsymbol{\omega })\underset{d}{=}\sqrt{a}b^{d/2}W(d\mu,d\boldsymbol{\omega }),\quad \forall \mu\in \mathbb{R},\ \boldsymbol{\omega}\in \mathbb{R}^{d},\label{swnrm}
 \end{eqnarray}
 \noindent and the It$\widehat{\mbox{o}}$ formula (see, e.g., Dobrushin and Major \cite{Dobrushin}; Major \cite{Major81}), from equation (\ref{wtr}), we obtain
 \begin{eqnarray}&&
 Y_{2,T} =\frac{\int_{0}^{T}\int_{\mathcal{K}}\left(Z^{2}(\mathbf{x},t)-1\right)d\mathbf{x}dt}{T^{1-\alpha }\mathcal{L}(T)\sqrt{c_{\mathcal{K}}(2,\alpha )}}\nonumber\\
 &&\underset{d}{=} \int_{\mathbb{R}^{2}}^{\prime }\left[\int_{0}^{1}\exp\left(i(\mu_{1}+\mu_{2})t\right)dt\right]\left[\prod_{j=1}^{2}f_{\mathcal{T}}
 \left(\frac{\mu_{j}}{T}\right)\right]^{1/2}\frac{1}{T^{1-\alpha }\mathcal{L}(T)\sqrt{c_{\mathcal{K}}(2,\alpha )}}\nonumber\\
  & &\times\int_{\mathbb{R}^{2d}}^{\prime }\left[\int_{\mathcal{K}}\exp\left(i\left\langle\mathbf{x},\boldsymbol{\omega }_{1}+\boldsymbol{\omega}_{2}\right\rangle\right)d\mathbf{x}\right]\left[\prod_{j=1}^{2}f_{\mathcal{S}}(\boldsymbol{\omega}_{j})\right]^{1/2}
   W\left(d\mu_{1}, d\boldsymbol{\omega}_{1}\right)W\left(d\mu_{2}, d\boldsymbol{\omega}_{2}\right).\nonumber\\ \label{wir}
 \end{eqnarray}

 We  denote
 \begin{eqnarray}
   &&\mathcal{I}_{\mathcal{K}}=\frac{1}{c_{\mathcal{K}}(2,\alpha )}\int_{\mathbb{R}^{2d}}\left|\int_{\mathcal{K}}\exp\left(i\left\langle\mathbf{x},\boldsymbol{\omega }_{1}+\boldsymbol{\omega}_{2}\right\rangle\right)d\mathbf{x}\right|^{2}
 \left[\prod_{j=1}^{2}f_{\mathcal{S}}(\boldsymbol{\omega}_{j})\right]
   d\boldsymbol{\omega}_{1}d\boldsymbol{\omega}_{2}\nonumber\\
   &&=  \frac{|\mathcal{K}|^{2}}{c_{\mathcal{K}}(2,\alpha )}\mathbb{E}\left[C^{2}_{\mathcal{S}}\left(\|P_{1}-P_{2}\|\right)\right]\nonumber\\
       &&=\frac{|\mathcal{K}|^{2}}{c_{\mathcal{K}}(2,\alpha )}\int_{0}^{\mathcal{D}(\mathcal{K})}C^{2}_{\mathcal{S}}\left(z\right)\psi_{\rho_{\mathcal{K}}}(z)dz,
   \label{si}\end{eqnarray}
 \noindent where in the last identity, we have applied similar steps to (\ref{rvex}).

   From  (\ref{wir})  and  (\ref{si}),    we  then obtain
 %%%%%%%%%%%%%%ERROR EN ESTA PARTE PREGUNTAR%%%%%%%%%%%%%%%%%%%
 \begin{eqnarray}
 &&\mathbb{E}\left[Y_{2,T}-\mathcal{R}
 \right]^{2}=[c_{T}(\alpha )]^{2}\mathcal{I}_{\mathcal{K}}\int_{\mathbb{R}^{2}}\left|\int_{0}^{1}\exp\left(i(\mu_{1}+\mu_{2})t\right)dt\right|^{2}
 \nonumber\\
 &&\hspace*{4cm}\times \frac{1}{|\mu_{1}\mu_{2}|^{1-\alpha }}Q_{T}(\mu_{1},\mu_{2})d\mu_{1}d\mu_{2},
 \label{cmq}
 \end{eqnarray}
 \noindent where
 \begin{eqnarray}
 Q_{T}(\mu_{1},\mu_{2})=\left(\frac{1}{c_{T}(\alpha )}|\mu_{1}\mu_{2}|^{\frac{1-\alpha}{2} }\prod_{j=1}^{2}f_{\mathcal{T}}^{1/2}\left(\frac{\mu_{j}}{T}\right)\frac{1}{T^{1-\alpha }\mathcal{L}(T)}-1\right)^{2}.\label{pcz}
 \end{eqnarray}

 Applying   Tauberian Theorems (see (\ref{TT})),   and Dominated Convergence Theorem,  as $T\to \infty,$ (\ref{cmq}) converges to zero for $\alpha \in (0,1/2).$
 Hence, the convergence in probability distribution of the random variable
 $Y_{2,T}$ to $\mathcal{R}$  holds (see Leonenko and Olenko \cite{LeonenkoOlenko14}, for more details).
 \end{proof}
 \subsection{Non--separable covariance functions}
 \label{nscf}
 Let $\varphi (v)\geq 0$  be a completely monotone function. That is, an infinite differentiable function satisfying
  $$(-1)^{n}\frac{d^{n}\varphi}{dv^{n}}(v)\geq 0,\quad v>0,\quad n\geq 0.$$

  By Bernstein's Theorem
  $$ \varphi (v)=\int_{0}^{v} \exp\left(-v\xi\right)\mu(d\xi),$$
  \noindent where $\mu $ is a positive measure over $[0,\infty)$ (see Gneiting \cite{Gneiting02}).
  %%%%%%%%%función generatriz de momentos de una v.a.cont. (t) +%%%%%%%%%%%

  Suppose further that $\psi :[0,\infty)\to[0,\infty)$ has completely monotone derivatives, i.e., it is a Bernstein function. The Gneiting class of spatiotemporal
  covariance functions is defined as follows (see Gneiting \cite{Gneiting02})
  \begin{eqnarray}
\widetilde{C}(\left\Vert \mathbf{x}-\mathbf{y}\right\Vert ,|t-s|)&=&\frac{1}{[\psi (|t-s|^{2})]^{d/2}}%
\varphi \left( \frac{\left\Vert \mathbf{x}-\mathbf{y}\right\Vert ^{2}}{\psi (|t-s|^{2})}\right) \nonumber\\
&=&C(z,\tau)=\frac{1}{[\psi (\tau ^{2})]^{d/2}}%
\varphi \left( \frac{z ^{2}}{\psi (\tau^{2})}\right)\nonumber\\
&& \quad \mathbf{x},\mathbf{y}\in \mathbb{R}^{d},\  t,s\in \mathbb{R
},  \ \tau, z\geq 0.\label{gneit}
\end{eqnarray}

It is known that  the one--parameter Mittag--Leffler function   $E_{\nu},$ for  \linebreak $0<\nu \leq 1,$
is a completely monotone function (see Feller \cite{Feller71}, p. 147), given by
\begin{equation*}
E_{\nu }(z)=\sum_{k=0}^{\infty}\frac{z^{k}}{
\Gamma (k\beta +1)} ,\quad z\in \mathbb{C},\quad 0<\beta <1
\end{equation*}
\noindent (see Erd\'elyi \emph{et al.} \cite{Erdelyi};  Haubold, Mathai and  Saxena
    \cite{Haubold}).

For every $\nu  \in (0,1),$ uniformly in  $x\in \mathbb{R}_{+},$ the following two--sided estimates are obtained  with optimal constants (see  Simon \cite{Simon2014}, Theorem 4):
\begin{eqnarray}
\frac{1}{1+\Gamma (1-\nu )x} &\leq & E_{\nu}(-x)\leq \frac{1}{1+[\Gamma (1+\nu)]^{-1}x}.
\label{eqimlf}
\end{eqnarray}

Note that the function
 \begin{equation*}
\psi (u)=(1+au^{\alpha })^{\beta },\text{ }a>0,\text{ }0<\alpha \leq 1,\text{ }0<\beta \leq 1,
\text{ }u\geq 0
\end{equation*}
\noindent has completely monotone derivatives (as well as the functions, for $b>1,$  $\psi_{2}(u)=\frac{\log(b+au^{\alpha })}{log(b)},$ and
$\psi_{3}(u)=\frac{(b+au^{\alpha })}{b(1+au^{\alpha })},$ for $0<b\leq 1$). Thus, we consider the Gneiting class of covariance functions
\begin{eqnarray}
C_{Z}(z ,\tau )&=&\frac{1}{(a\tau^{2\alpha }+1)^{\beta d/2}}E_{\nu }\left( -\frac{z ^{2\gamma }}{(a \tau  ^{2\alpha }+1)^{\beta \gamma }}\right). \nonumber\\
&& z,\tau\geq 0,\quad \nu ,\alpha ,\beta ,\gamma \in (0,1),\quad a>0. \label{GCCF}
\end{eqnarray}
From (\ref{eqimlf}), the following proposition is derived.
\begin{proposition}
\label{pr1} Under \textbf{Condition 1}, and for the Gneiting class of covariance functions introduced in (\ref{GCCF}), \textbf{Condition 2(ii)}
holds if $m=1,$ for  $0<2\alpha \beta  (d/2-\gamma )<1,$  and for $0<2\alpha \beta  (d/2-\gamma )<1/2$ if $m=2.$
\end{proposition}
\begin{proof}
The proof follows straightforward from  equations (\ref{eqimlf}) and  (\ref{GCCF}). Specifically, for $m=1,$
\begin{eqnarray}&&\nonumber\\
&&\sigma_{1,\mathcal{K}}^{2}(T)=2T^{2}|\mathcal{K}|^{2}\int_{[0,1]}(1-\tau)\int_{0}^{\mathcal{D}(\mathcal{K})}C_{Z}(z ,T\tau )\psi_{\rho_{\mathcal{K}}}(z)dzd\tau\nonumber\\
&&=2T^{2}|\mathcal{K}|^{2}\int_{[0,1]}(1-\tau)\int_{0}^{\mathcal{D}(\mathcal{K})}\frac{1}{(a[T\tau]^{2\alpha }+1)^{\beta d/2}}
\nonumber\\
&&\hspace*{1cm}\times
E_{\nu }\left( -\frac{z ^{2\gamma }}{(a [T\tau ] ^{2\alpha }+1)^{\beta \gamma }}\right)\psi_{\rho_{\mathcal{K}}}(z)dzd\tau\nonumber\\
& &\geq 2T^{2}|\mathcal{K}|^{2}\int_{[0,1]}(1-\tau)\int_{0}^{\mathcal{D}(\mathcal{K})}\frac{1}{(a[T\tau]^{2\alpha }+1)^{\beta d/2}}
\nonumber\\
&&\hspace*{3cm}\times
\frac{1}{\left[ 1+\Gamma (1-\nu)\frac{z^{2\gamma }}{\left[1+aT^{2\alpha}\tau ^{2\alpha}\right]^{\beta \gamma}}\right]}\psi_{\rho_{\mathcal{K}}}(z)dzd\tau
\nonumber\end{eqnarray}
\begin{eqnarray}
&&\geq 2T^{2}|\mathcal{K}|^{2}\int_{[0,1]}(1-\tau)\int_{0}^{\mathcal{D}(\mathcal{K})}\frac{1}{(a[T\tau]^{2\alpha }+1)^{\beta d/2}}
\nonumber\\
&&\hspace*{3cm}\times\frac{a^{\beta\gamma}T^{2\alpha \beta \gamma }\tau^{2\alpha \beta \gamma } }{\left[ \left[1+aT^{2\alpha}\tau ^{2\alpha}\right]^{\beta \gamma}+\Gamma (1-\nu)z^{2\gamma }\right]}\psi_{\rho_{\mathcal{K}}}(z)dzd\tau
\nonumber\\
&&=2T^{2\left(1-\alpha \beta \left(\frac{d}{2}-\gamma \right)\right)}|\mathcal{K}|^{2}
\int_{[0,1]}(1-\tau)\int_{0}^{\mathcal{D}(\mathcal{K})}\frac{1}{\left(a\tau^{2\alpha }+\frac{1}{T^{2\alpha}}\right)^{\beta d/2}}
\nonumber\\
&&\hspace*{3cm}\times\frac{a^{\beta\gamma}\tau^{2\alpha \beta \gamma } }{\left[ \left[1+aT^{2\alpha}\tau ^{2\alpha}\right]^{\beta \gamma}+\Gamma (1-\nu)z^{2\gamma }\right]}\psi_{\rho_{\mathcal{K}}}(z)dzd\tau.
\nonumber\\
\label{pde2h}
\end{eqnarray}

From (\ref{pde2h}), \textbf{Condition(ii)} holds for $\alpha \beta \left(\frac{d}{2}-\gamma \right) <1.$

 In a similar way to (\ref{pde2h}), it can be proved that for $m=2,$ \textbf{Condition(ii)} also holds for $\alpha \beta \left(\frac{d}{2}-\gamma \right)<1/2.$
\end{proof}

\medskip

As a direct consequence of Proposition \ref{pr1},  we obtain that Theorems \ref{th1}  and \ref{th2}  hold for the family of spatiotemporal Gaussian random fields with covariance function (\ref{GCCF}).

Similar assertions hold for the family of spatiotemporal covariance functions
\begin{eqnarray}&&
\widetilde{C}_{Z}(\mathbf{z},\tau )=\frac{\sigma ^{2}}{[\psi (\tau ^{2})]^{d/2}}%
\varphi \left( \frac{\left\Vert \mathbf{z}\right\Vert ^{2}}{\psi (\tau ^{2})}\right),
\ \sigma ^{2}\geq 0,\ (\mathbf{z},\tau )\in \mathbb{R}^{d}\times \mathbb{R}\nonumber\\&&
\varphi (u)=\frac{1}{(1+cu^{\gamma })^{\nu }},\ u\geq 0,\ c>0, \ 0<\gamma \leq 1, \ \nu >0\nonumber\\
\psi (u)&=&(1+au^{\alpha })^{\beta },\text{ }a>0,\ 0<\alpha \leq 1,\ 0<\beta \leq 1,\ u\geq 0,\label{fcf}
\end{eqnarray}
\noindent for $\alpha \beta \left(\frac{d}{2}-\gamma \nu \right) <1$ if $m=1,$ and for $\alpha \beta \left(\frac{d}{2}-\gamma \nu \right)<1/2$
if $m=2.$

\section{Discussion}
As commented in the Introduction, the main contribution of this paper relies on  deriving a general reduction principle in  Theorem \ref{th3},  beyond the regularly varying condition on
the spatiotemporal  covariance function of the underlying Gaussian random field.  Hence, we can analyze  a larger class of spatiotemporal covariance functions. Particularly, some examples of    Gneiting  class are considered (see Gneiting \cite{Gneiting02}). This class of covariance functions is popular in many applications, including Meteorology or Earth sciences, among  others.

By considering homogeneous and isotropic Gaussian random fields restricted to a spatial convex compact
set  evolving over time, this paper applies an extrinsic  random field approach, alternatively to the intrinsic spherical one  adopted in  Marinucci, Rossi and Vidotto \cite{MarinucciRV}. Thus, the isonormal representation of the underlying spatiotemporal Gaussian random field on  $\mathbb{R}^{d}\times \mathbb{R},$ and the characteristic function of the uniform probability distribution on a temporal  interval and a spatial convex compact set allow the consideration of a continuous spectral based approach, in the derivation of limit  results  in our framework (see, e.g., Proposition \ref{p41}).

A time--varying pure point spectral approach is considered in Marinucci, Rossi and Vidotto \cite{MarinucciRV}, based on projection onto the orthonormal basis of spherical harmonics. Different ranges of dependence are then assumed at different  spatial resolution levels in the sphere. In our paper,  under the temporal decay velocity of the space--time covariance function established in \textbf{Condition 2}, a general reduction principle is derived in  Theorem \ref{th3}, providing the limiting distribution of properly normalised integrals of non-linear transformations of spatiotemporal Gaussian random fields. Theorem \ref{th1} constitutes a  particular case of Theorem \ref{th3}, where the scaling also depends on the threshold $u,$  that provides a similar scenario to Theorem 1 in Marinucci, Rossi and Vidotto \cite{MarinucciRV}, when zero-th order multipole component is long--memory, and all the other multipoles have asymptotically smaller variance.

Note also that Proposition \ref{pr1} of this paper corresponds to the separable case in time and space which is different situation to the non--separable case addressed in Marinucci, Rossi and Vidotto \cite{MarinucciRV}. Furthermore,  Marinucci, Rossi and Vidotto \cite{MarinucciRV} consider  different orthonormal bases for space and time, respectively in terms of the spherical harmonics and the complex exponentials. These bases do not provide a  diagonal spectral representation of the space--time covariance function of the underlying Gaussian random field. The composite Rosenblatt distribution then arises from the set of multipoles where the larger dependence range (long--memory) is displayed (as in reduction theorems). This non--diagonal representation induces a similar effect to considering Hermite rank $m=2,$ in the case of separable covariance functions in space and time, admitting a diagonal representation in terms of the complex exponentials in space and time (see Proposition \ref{p41}).

In a subsequent paper, our results can be extended to the non--separable case, in terms of  a bounded spatially varying long--memory parameter satisfying \textbf{Condition 2}(ii). This could be the case, for example,  of an extended version of Proposition \ref{p41}, in terms of non--separable covariance functions,  involving a bounded spatial  frequency  varying $\alpha (\cdot )$ parameter. In that case, when the  supremum of $\alpha (\cdot )$ over the spatial frequencies satisfies  \textbf{Condition 2}(ii), Theorem \ref{th3} holds under  \textbf{Condition 2}(i).

\vspace*{1cm}

\noindent \textbf{Acknowledgements}. \footnotesize{ This work has been supported in part by projects  MCIN/ AEI/PGC2018-099549-B-I00, CEX2020-001105-M MCIN/AEI/10.13039/501100011033, \linebreak and  by grant A-FQM-345-UGR18 cofinanced by ERDF Operational Programme 2014-2020, and the Economy and Knowledge Council of the Regional Government of Andalusia, Spain. N. Leonenko  was partially supported by LMS grant 42997, ARC grant DP220101680. N. Leonenko  and M.D. Ruiz--Medina where supported the Isaak Newton Institute (Cambridge) Program \emph{Fractional Differential Equations}.

}

%%%%%%%%%revisar con mis resultados%%%%%%%%%%%%%%%%%%%%%%%%%
\end{document}